\documentclass{amsart}

\usepackage{hyperref}

\usepackage{amssymb}
\usepackage{amsmath}
\usepackage{amsthm}
\usepackage{enumerate}
\usepackage{graphicx}
\usepackage{tikz-cd}

\theoremstyle{plain}

\makeatletter
\newtheorem*{rep@theorem}{\rep@title}
\newcommand{\newreptheorem}[2]{%
\newenvironment{rep#1}[1]{%
 \def\rep@title{#2 \ref{##1}}%
 \begin{rep@theorem}}%
 {\end{rep@theorem}}}
\makeatother

\newtheorem{theorem}{Theorem}[section]
\newreptheorem{theorem}{Theorem}

\newtheorem{lemma}[theorem]{Lemma}
\newtheorem{corollary}[theorem]{Corollary}
\newtheorem{question}[theorem]{Question}

\newtheorem{problem}[theorem]{Problem}
\newtheorem{observation}[theorem]{Observation}

\theoremstyle{definition}
\newtheorem{definition}[theorem]{Definition}
\newtheorem{example}[theorem]{Example}

\theoremstyle{remark}
\newtheorem*{remark}{Remark}

\newtheorem*{acknowledgments}{Acknowledgments}

\newcommand{\UnitDisk}{\mathbb{D}}
\newcommand{\ComplexPlane}{\mathbb{C}}
\newcommand{\Reals}{\mathbb{R}}
\newcommand{\RiemannSphere}{\widehat{\mathbb{C}}}

\newcommand{\eps}{\varepsilon}

\DeclareMathOperator{\id}{id}
\DeclareMathOperator{\spann}{span}

\renewcommand{\Re}{\mathrm{Re\,}}

\begin{document}

\title{Rational Ahlfors functions}

\date{March 11, 2014}

\author[M. Fortier Bourque]{Maxime Fortier Bourque}
\thanks{First author supported by NSERC}
\address{Department of Mathematics, Graduate Center, City University of New York, 365 Fifth Avenue, New York (NY), USA, 10016.}
\email{maxforbou@gmail.com}

\author[M. Younsi]{Malik Younsi}
\thanks{Second author supported by the Vanier Canada Graduate Scholarships program.}
\address{D\'epartement de math\'ematiques et de statistique, Pavillon Alexandre--Vachon, $1045$ av. de la M\'edecine, Universit\'e Laval, Qu\'ebec (Qu\'ebec), Canada, G1V 0A6.}
\email{malik.younsi.1@ulaval.ca}

\keywords{Analytic capacity, Ahlfors functions, rational maps, conformal representation}
\subjclass[2010]{primary 30C85, 30C20; secondary 30F10, 65E05}

\begin{abstract}
We study a problem of Jeong and Taniguchi asking to find all rational maps which are Ahlfors functions. We prove that the rational Ahlfors functions of degree two are characterized by having positive residues at their poles. We then show that this characterization does not generalize to higher degrees, with the help of a numerical method for the computation of analytic capacity. We also provide examples of rational Ahlfors functions in all degrees.
\end{abstract}

\maketitle

\section{Introduction}

Let $X$ be a domain in the Riemann sphere containing $\infty$. The derivative at infinity of a holomorphic function $f: X \to \ComplexPlane$ is defined as
$$
f'(\infty):=\lim_{z \to \infty} z(f(z)-f(\infty)).
$$
Among all the holomorphic functions $f : X \to \UnitDisk$ satisfying $f(\infty)=0$, there is a unique one which maximizes $\Re f'(\infty)$, called the \textit{Ahlfors function} (see \cite{HAV} or \cite{FISH}).

The Ahlfors function is constant precisely when the complement $\RiemannSphere \setminus X$ is remo\-vable for bounded holomorphic functions. This is what motivated Ahlfors to study the above extremal problem in his seminal paper \cite{AHL}. The main result of \cite{AHL} can be viewed as a generalization of both the Schwarz lemma and the Riemann mapping theorem.

\begin{theorem}[Ahlfors \cite{AHL}]
\label{theo Ahlfors}
Let $X$ be a non-degenerate $n$-connected domain containing $\infty$. Then the Ahlfors function $f:X \to \UnitDisk$ is a proper holomorphic map of degree $n$.
\end{theorem}

Using Ahlfors' theorem, Jeong and Taniguchi proved that up to a biholomorphism, the Ahlfors function on a finitely connected domain is simply the restriction of a rational map.

\begin{theorem}[Jeong-Taniguchi \cite{JEONG}]
\label{JT}
Let $X$ be a non-degenerate $n$-connected domain containing $\infty$, and let $f$ be the Ahlfors function on $X$. Then there exists a rational map $R$ of degree $n$ and a biholomorphism $g : X \to R^{-1}(\UnitDisk)$ such that $f = R \circ g$.
\end{theorem}

By applying a M\"obius transformation, we may assume that the biholomorphism $g$ satisfies $\lim_{z \to \infty} (g(z)-z) = 0$. Then, by a simple change of variable argument, the rational map $R$ is the Ahlfors function on $R^{-1}(\UnitDisk)$, and we say that it is a \textit{rational Ahlfors function}.

With this normalization, $R$ admits the partial fractions decomposition
$$
R(z) = \sum_{j=1}^n \frac{a_j}{z-p_j}
$$
for some $a_1,...,a_n \in \ComplexPlane \setminus \{0\}$ and distinct $p_1,...,p_n \in \ComplexPlane$.

The following problem was posed by Jeong and Taniguchi in \cite[Problem 1.5]{JEONG}.

\begin{problem} Determine which residues $a_1,...,a_n$ and poles $p_1,...,p_n$ correspond to rational Ahlfors functions.
\end{problem}

The problem is trivial for $n=1$ : its solution is given by $a_1 > 0$ and $p_1 \in \ComplexPlane$. The solution for $n=2$ can be expressed in equally simple terms.

\begin{theorem}
\label{deg2}
A rational map of degree two is a rational Ahlfors function if and only if it can be written as
$$R(z)= \frac{a_1}{z-p_1}+\frac{a_2}{z-p_2},$$
for distinct $p_1, p_2 \in \mathbb{C}$ and positive $a_1,a_2$ satisfying $a_1+a_2 < |p_1-p_2|$.
\end{theorem}

As observed in \cite{JEONG}, the locus of parameters $a_j$ and $p_j$ corresponding to rational Ahlfors functions of degree $n$ has dimension $3n$. This is because the moduli space of a sphere with $n$ disks removed and one point marked has dimension $(3n-4)$, and there is a redundancy of representation due to the action of the $4$-dimensional group of complex affine transformations. In terms of dimensions, it thus seemed possible at first that the above theorem could generalize to higher degrees, but this is not the case.

\begin{theorem}
\label{3andup}
For every $n \geq 3$, the positivity of the residues of a rational map of degree $n$
is neither sufficient nor necessary for it to be a rational Ahlfors function.
\end{theorem}

However, in any degree, there exist many rational Ahlfors functions with po\-sitive residues. Here is a $2n$-dimensional family of examples of degree $n$.

\begin{theorem}
\label{real}
Let $R$ be a rational map of degree $n$ with only real poles and real residues such that $R^{-1}(\UnitDisk)$ is $n$-connected. Then $R$ is a rational Ahlfors function if and only if all of its residues are positive.
\end{theorem}

The rest of the paper is organized as follows. In section 2, we give a proof of Theorem \ref{JT} and supplement it with a uniqueness statement. Section 3 is devoted to the proof of Theorem \ref{real}. A family of rational Ahlfors functions with rotational symmetry is exhibited in section 4. Section 5 contains the proof of Theorem \ref{deg2}.  In section 6, we define analytic capacity and summarize a method due to the second author and Ransford \cite{YOU} for its computation. We illustrate the method with various rational Ahlfors functions, and give an example of a rational map of degree $3$ with positive residues which is not Ahlfors. The example is then generalized to higher degrees. Finally, in section 7, we prove that the residues of a rational Ahlfors function are not necessarily positive.

\section{The Bell representation theorem}

\subsection{Restrictions of rational maps}

We are interested in rational maps $R$ of degree $n$ such that the open set
$$R^{-1}(\UnitDisk)= \{\, z \in \RiemannSphere \,:\, |R(z)|<1 \,\}$$
is an $n$-connected domain. A version of the following lemma appears in \cite[Lemma 2.7]{JEONG2}.

\begin{lemma} \label{nconnected}
Let $R$ be a rational map of degree $n$. The following are equivalent :
\begin{enumerate}
\item[\emph{(1)}] $R^{-1}(\RiemannSphere \setminus \UnitDisk)$ has at least $n$ connected components;
\item[\emph{(2)}] $R^{-1}(\UnitDisk)$ and $R^{-1}(\RiemannSphere \setminus \UnitDisk)$ have $1$ and $n$ connected components respectively;
\item[\emph{(3)}] $R^{-1}(\UnitDisk)$ is connected and bounded by $n$ disjoint analytic Jordan curves;
\item[\emph{(4)}] $R$ maps each component of $R^{-1}(\RiemannSphere \setminus \UnitDisk)$ homeomorphically onto $\RiemannSphere \setminus \UnitDisk$;
\item[\emph{(5)}] all the critical values of $R$ are in $\UnitDisk$.
\end{enumerate}
\end{lemma}

\begin{proof}
The implications $(3)\Rightarrow(2)\Rightarrow(1)$ are trivial. For the rest of the proof, we let $X:=R^{-1}(\UnitDisk)$ and $Y:=R^{-1}(\RiemannSphere \setminus \UnitDisk)$.

\paragraph{$(1) \Rightarrow (4)$}
Suppose that $Y$ has $m$ connected components, say $K_1, ..., K_m$, where $m\geq n$. Since $R$ is open, $K_j$ is compact, and $\RiemannSphere \setminus \UnitDisk$ is connected,  the restriction $R : K_j \to \RiemannSphere \setminus \UnitDisk$ is surjective for every $j$. Therefore, $m=n$ and $R$ is injective on each $K_j$, for otherwise some point in $\RiemannSphere \setminus \UnitDisk$ would have more than $n$ preimages by $R$. Thus each restriction $R: K_j \to \RiemannSphere \setminus \UnitDisk$ is continuous, open and bijective and hence a homeomorphism.

\paragraph{$(4) \Rightarrow (5)$}
Suppose that $R$ maps each component of $Y$ homeomorphically onto $\RiemannSphere \setminus \UnitDisk$. In particular, $R$ is locally injective on $Y$ and thus has no critical points in the interior of $Y$.  If $R$ has a critical point $p$ of order $m\geq 1$ on $\partial Y$, the topological graph $\partial Y= R^{-1}(\partial\UnitDisk)$ has a vertex of order $2(m+1) > 2$ at $p$. This is because $R$ is locally like $z\mapsto z^{m+1}$ near $p$. But for each component $K$ of $Y$, the restriction $R : \partial K \to \partial \UnitDisk$ is a homeomorphism, so that $\partial Y$ is a union of disjoint Jordan curves. It follows that all the critical points of $R$ are in $X$, so that all the critical values of $R$ are in $\UnitDisk$. This proves $(4) \Rightarrow (5)$.

\paragraph{$(5) \Rightarrow (3)$} If all the critical values of $R$ are in $\UnitDisk$, then all its critical points are in $X$. The restriction $R : Y \to \RiemannSphere \setminus \UnitDisk$ is thus an $n$-sheeted unbranched covering map. Since the closed disk $\RiemannSphere \setminus \UnitDisk$ is simply connected, $Y$ has exactly $n$ connected components, each homeomorphic to $\RiemannSphere \setminus \UnitDisk$. The complement $X$ of these $n$ disjoint Jordan disks in the Riemann sphere is necessarily connected (see e.g. \cite[11.4, p.117]{NEW}). Moreover, as $R : \partial X \to \partial\UnitDisk$ is locally bi-analytic, each of the $n$ Jordan curves bounding $X$ is analytic.
\end{proof}

A rational map $R$ of degree $n$ which satisfies any (and hence all) of the conditions of Lemma \ref{nconnected}, and satisfies in addition $R(\infty)=0$, will be called \textit{$n$-good}. An $n$-good rational map must have only simple poles, since it cannot have $\infty$ as a critical value. Every $n$-good rational map $R$ can thus be written as a sum of partial fractions of the form
$$
R(z)=\sum_{j=1}^n \frac{a_j}{z-p_j}
$$
for some $a_1,...,a_n \in \ComplexPlane \setminus \{ 0 \}$ and distinct $p_1,...,p_n \in \ComplexPlane$, and this representation is unique up to permutation of the factors.

Observe that if $R$ is $n$-good, then $R^{-1}(\UnitDisk)$ is a non-degenerate $n$-connected domain and the restriction $R : R^{-1}(\UnitDisk) \to \UnitDisk$ is a proper holomorphic map of degree $n$. In the next subsection, we will see that conversely, if $X$ is a non-degenerate $n$-connected domain and $f : X \to \UnitDisk$ is a proper holomorphic map of degree $n$, then up to a biholomorphism, $f$ is an $n$-good rational map.

\subsection{The Bell representation theorem}
\label{sec1}

The existence part of the following theo\-rem is due to Jeong and Taniguchi \textbf{\cite{JEONG}}. It was dubbed ``the Bell representation theorem'', as it was anticipated by Steven Bell. The uniqueness statement is new.

\begin{theorem} \label{Bellrep}
Let $X$ be a non-degenerate $n$-connected domain and let $f: X \to \UnitDisk$ be a proper holomorphic map of degree $n$. Then there exists a rational map $R$ of degree $n$ and a biholomorphism $g : X \to R^{-1}(\UnitDisk)$ such that $f = R \circ g$. Moreover, if $f=Q \circ h$ for another rational map $Q$ of degree $n$ and a biholomorphism $h : X \to Q^{-1}(\UnitDisk)$, then there is a M\"obius transformation $M$ such that $Q=R\circ M^{-1}$ and $h=M \circ g$.
\end{theorem}

\begin{proof}
By applying the Riemann mapping theorem $n$ times, we may assume that $X$ is bounded by $n$ disjoint analytic Jordan curves.

Then the map $f : X \to \UnitDisk$ extends analytically to a neighborhood of $\overline X$ by the Schwarz reflection principle. Let $E$ be a connected component of $\partial X$. The extension $f : E \to \partial\UnitDisk$ is surjective, and thus injective, since $f$ has degree $n$ on $\overline X$ and $\partial X$ has $n$ connected components. Moreover, $f$ has no critical points on $E$. Indeed, if $f$ has a critical point at $p \in E$, then $f$ behaves like $z \mapsto z^m$ for some $m>1$ near $p$, so that some arc in $X$ ending at $p$ is mapped to $\partial \UnitDisk$ by $f$. This is impossible since $f$ maps $X$ into $\UnitDisk$. The inverse map $f^{-1} : \partial\UnitDisk \to E$ is thus analytic.

We next construct a Riemann surface $Y$ by welding $n$ disks to $\overline X$ using the boundary extension of $f$. Label the components of $\partial X$ as $E_1,...,E_n$, let $D_1,...,D_n$ be disjoint copies of the disk $\RiemannSphere \setminus \UnitDisk$, and let $\pi_j : D_j \to \RiemannSphere \setminus \UnitDisk$ be the projection map which forgets the label. We define
$$
Y := \overline X \sqcup_{j=1}^n D_j / \sim,
$$
where $\sim$ identifies $z \in E_j$ with $w \in \partial D_j$ if $f(z)=\pi_j(w)$. Since the gluing homeomorphism $\pi_j^{-1} \circ f : E_j \to \partial D_j$ is bi-analytic for each $j$, $Y$ can be equipped with a complex atlas such that the inclusions $\overline X \hookrightarrow Y$ and $D_j \hookrightarrow Y$ are analytic.

Next, observe that $f$ can be extended to an analytic map $Y \to \RiemannSphere$ of degree $n$ by
$$
z \mapsto \begin{cases} f(z) & \text{if }z\in \overline{X} \\ \pi_j(z) & \text{if } z \in D_j.  \end{cases}
$$
We still denote this extension by $f$.

Since $Y$ is topologically a sphere, the uniformization theorem implies that there is a biholomorphism $g : Y \to \RiemannSphere$. The composition $f\circ g^{-1} : \RiemannSphere \to \RiemannSphere$ is analytic, and hence a rational map $R$, of the same degree $n$ as $f$. By construction, we have $f^{-1}(\UnitDisk)=X$, so that $g(X)=R^{-1}(\UnitDisk)$. This proves the existence part.

Suppose that $f$ factors in another way as $Q \circ h$ on $X$, where $Q$ is a rational map of degree $n$ and $h : X \to Q^{-1}(\UnitDisk)$ is a biholomorphism. Note that $R^{-1}(\UnitDisk)$ and $Q^{-1}(\UnitDisk)$ are $n$-connected domains so that Lemma \ref{nconnected} applies. Let $K_1,...,K_n$ be the connec\-ted components of $R^{-1}(\RiemannSphere \setminus \UnitDisk)$. Also let $L_1,...,L_n$ be the components of $Q^{-1}(\RiemannSphere \setminus \UnitDisk)$, labelled in such a way that the boundary extension of the biholomorphism
$$
h \circ g^{-1} :R^{-1}(\UnitDisk) \to Q^{-1}(\UnitDisk)
$$
maps $\partial K_j$ onto $\partial L_j$.

By Lemma \ref{nconnected}, $R$ maps each $K_j$ homeomorphically onto $\RiemannSphere \setminus \UnitDisk$ and similarly for $Q$. We define a map $M: \RiemannSphere \to \RiemannSphere$ by
$$
M(z) := \begin{cases} h \circ g^{-1}(z)  &  \text{if }z\in R^{-1}(\UnitDisk) \\ Q^{-1} \circ R(z) & \text{if }z \in R^{-1}(\RiemannSphere \setminus \UnitDisk), \end{cases}
$$
where the appropriate branch of $Q^{-1} \circ R$ is chosen so that it maps $K_j$ to $L_j$. Then $M$ is a homeomorphism and is conformal on the complement of $\partial R^{-1}(\UnitDisk)$. By Morera's theorem, analytic curves are removable for continuous maps holomorphic in their complement. Therefore, $M$ is conformal on all of $\RiemannSphere$ and hence a M\"obius transformation. By construction, we have $Q\circ M = R$ and $M \circ g = h$.
\end{proof}

If we require that $g$ in the above theorem satisfies a certain normalization, then the factorization $f=R \circ g$ becomes unique. In what follows, a meromorphic function $h$ defined in a neighborhood of $\infty$ in $\RiemannSphere$ is said to be \textit{normalized at infinity} if it satisfies
$$
\lim_{z \to \infty} h(z) - z = 0.
$$
Equivalently, $h$ is normalized at infinity if it
has a Laurent development of the form
$$
h(z)=z+\frac{b_1}{z}+\frac{b_2}{z^2}+\dots
$$
near $\infty$. The only M\"obius transformation which is normalized at infinity is the identity map.

\begin{corollary} \label{normalizedBellrep}
Let $X$ be a non-degenerate $n$-connected domain containing $\infty$ and let $f: X \rightarrow \UnitDisk$ be a proper holomorphic map of degree $n$ satisfying $f(\infty)=0$. Then there exists a unique $n$-good rational map $R$ and a unique biholomorphism $g : X \to R^{-1}(\UnitDisk)$ normalized at infinity such that $f=R\circ g$.
\end{corollary}

We call the factorization $f=R \circ g$ provided by Corollary \ref{normalizedBellrep} the \textit{normalized factorization of $f$}.

Given a non-degenerate $n$-connected domain $X$ containing $\infty$, one might wonder how many proper holomorphic maps $f : X \to \UnitDisk$ of degree $n$ satisfying $f(\infty)=0$ there are. The answer is contained in the following theorem, proofs of which can be found in \cite[Theorem 2.2]{BELL} or \cite[Theorem 3]{KHA}.

\begin{theorem}[Bieberbach \cite{BIEB}, Grunsky \cite{GRU}]
\label{normalizedBieberbach}
Let $X$ be a domain containing $\infty$ and bounded by $n$ disjoint analytic Jordan curves $E_1,...,E_n$. For each $j$, let $\alpha_j$ be any point in $E_j$. Then there exists a unique proper holomorphic map $f : X \to \UnitDisk$ of degree $n$ satisfying $f(\infty)=0$ whose continuous extension to $\partial X$ maps $\alpha_j$ to $1$ for each $j$.
\end{theorem}

\subsection{Rational Ahlfors functions}

Let $X$ be a non-degenerate $n$-connected domain containing $\infty$ and let $f: X \rightarrow \UnitDisk$ be the Ahlfors function on $X$. This means that
$$
\Re h'(\infty) \leq \Re f'(\infty)
$$
for every holomorphic map $h : X \to \UnitDisk$ satisfying $h(\infty)=0$. By Ahlfors' theorem, $f$ is proper of degree $n$. In virtue of Corollary \ref{normalizedBellrep}, $f$ admits a normalized factorization $f=R \circ g$, where $R$ is an $n$-good map and $g$ is a biholomorphism normalized at infinity.

Then $R$ is the Ahlfors function on $g(X)=R^{-1}(\UnitDisk)$. Indeed, for any holomorphic function $h : R^{-1}(\UnitDisk) \to \UnitDisk$ with $h(\infty)=0$, the composition $h \circ g : X \to \UnitDisk$ is holomorphic and vanishes at infinity, so that
$$
\Re h'(\infty) = \Re (h\circ g)'(\infty) \leq  \Re f'(\infty) = \Re (R \circ g)'(\infty) = \Re R'(\infty).
$$

\begin{definition}
A \textit{rational Ahlfors function} is an $n$-good rational map $R$ which is the Ahlfors function on $R^{-1}(\UnitDisk)$.
\end{definition}

The above change of variable can be done in more generality. Let $f : X \to \UnitDisk$ be the Ahlfors function on any domain $X$ containing $\infty$, and let $g : Y \to X$ be a biholomorphism with $g(\infty)=\infty$ and $\lim_{z \to \infty} g(z)/z = a$. Then the Ahlfors function on $Y$ is given by
$$
z \mapsto (a/|a|)f \circ g(z).
$$
We refer to this as the \textit{transformation law} for Ahlfors functions. A particular instance of this law which is relevant to us is the following.

\begin{observation} \label{affine}
If $R$ is a rational Ahlfors function, then so is
$$
z \mapsto (a/|a|)R(az+b)
$$
for every $a \in \ComplexPlane \setminus \{ 0 \}$ and every $b \in \ComplexPlane$.
\end{observation}

We give examples of rational Ahlfors functions in the next two sections.

\section{Reflection symmetry}
\label{sec2}

In this first family of examples, the domain $R^{-1}(\UnitDisk)$ admits a reflection symmetry which allows us to show that the Ahlfors function $f$ on $R^{-1}(\UnitDisk)$ takes the value $1$ where $R$ does. The uniqueness part of Theorem \ref{normalizedBieberbach} then allows us to conclude that $f=R$, so that $R$ is a rational Ahlfors function.

\begin{theorem}
\label{theo real}
Let
$$R(z):=\sum_{j=1}^{n} \frac{a_j}{z-p_j},$$
where the poles $p_1, \dots, p_n$ are distinct and real, and the residues $a_1, \dots, a_n$ are positive. If $R$ is $n$-good, then $R$ is a rational Ahlfors function.

\end{theorem}

We begin with two lemmas.

\begin{lemma}
\label{lem1}
Let $R$ be as in Theorem \ref{theo real}. Then $R$ is real-valued only on the real axis, which $\partial R^{-1}(\mathbb{D})$ intersects in $2n$ points $\alpha_1<\beta_1<\dots<\alpha_n<\beta_n$ such that $R(\alpha_j)=-1$ and $R(\beta_j)=1$ for each $j$.
\end{lemma}

\begin{proof}
Suppose that the poles are ordered such that $p_1<  \dots <p_n$.

Since $a_1,\dots,a_n,p_1,\dots,p_n \in \mathbb{R}$, we have $R(\Reals \cup \{\infty\}) \subset \Reals \cup\{ \infty \}$. One computes that
$$R'(z) = -\sum_{j=1}^n \frac{a_j}{(z-p_j)^2},$$
which is strictly negative on $\mathbb{R} \setminus \{ p_1,...,p_n \}$.

Thus, $R$ decreases from $0$ to $-\infty$ on $(-\infty,p_1)$, then decreases from $+\infty$ to $-\infty$ on $(p_j,p_{j+1})$ for $j=1,...,n-1$, and finally decreases from $+\infty$ to $0$ on $(p_n,+\infty)$. In particular, $R$ takes every real value exactly $n$ times on $\Reals$. Since $R$ has degree $n$, we have $R^{-1}(\Reals \cup \{\infty\}) = \Reals \cup\{ \infty \}$.

Since $\partial R^{-1}(\UnitDisk) = R^{-1}(\partial \UnitDisk)$, we have $(\partial R^{-1}(\UnitDisk)) \cap \Reals = R^{-1}(\{-1,1\})$. Because $R$ is decreasing on $\Reals \setminus \{ p_1,...,p_n \}$, the inverse images of $-1$ and $1$ intertwine as claimed.
\end{proof}

\begin{lemma}
\label{lem2}
Let $E$ be a finite union of disjoint closed intervals in the real line, say
$$E:= \cup_{j=1}^{n} [c_j,d_j],$$
where $c_j < d_j$ for each $j$. Then the Ahlfors function on $\RiemannSphere \setminus E$ is given by
$$z \mapsto \frac{\prod_{j=1}^{n}\left(\frac{z-c_j}{z-d_j} \right)^{1/2}-1}{\prod_{j=1}^{n}\left(\frac{z-c_j}{z-d_j} \right)^{1/2}+1}.$$
In particular, it is real-valued and has strictly negative derivative on $\Reals \setminus E$.
\end{lemma}

\begin{proof}
By a theorem of Pommerenke (see e.g. \cite[Chapter 1, Theorem 6.2]{GAR}), the derivative at infinity of the Ahlfors function on $\RiemannSphere \setminus E$ is equal to a quarter of the length of $E$. This allows us to construct the Ahlfors function explicitly.

On $\ComplexPlane \setminus E$, define
$$
g(z):=\frac{1}{2} \int_{E} \frac{dt}{z-t}
$$
and
$$
f(z):=\frac{e^{g(z)}-1}{e^{g(z)}+1}.
$$

An easy calculation shows that $|\operatorname{Im}g(z)|<\pi/2$ for $z \notin E$, and thus $f$ maps $\ComplexPlane \setminus E$ into the unit disk $\mathbb{D}$. Furthermore, we have $g(z),f(z) \to 0$ as $z\to \infty$ and
$$f'(\infty) = \frac{1}{2}g'(\infty)=\frac{1}{4}\int_E dt,$$ so that $f$ is the Ahlfors function on $\RiemannSphere \setminus E$ by Pommerenke's result.

A simple calculation yields
$$f(z)=\frac{\prod_{j=1}^{n}\left(\frac{z-c_j}{z-d_j} \right)^{1/2}-1}{\prod_{j=1}^{n}\left(\frac{z-c_j}{z-d_j} \right)^{1/2}+1},$$
where each square root is chosen to have image inside the right half-plane.

Clearly, $g$ and $f$ are real-valued on $\Reals \setminus E$. We have
$$
g'(z)= -\frac{1}{2} \int_E \frac{dt}{(z-t)^2},
$$
which is strictly negative on $\Reals \setminus E$. As the exponential and $z \mapsto (z-1)/(z+1)$ have strictly positive derivative on $\Reals$ and $\Reals_+$ respectively, we get that $f'<0$ on $\Reals \setminus E$.
\end{proof}

We can now proceed with the proof of Theorem \ref{theo real} :

\begin{proof}
Let $f$ be the Ahlfors function on the non-degenerate $n$-connected domain $X:=R^{-1}(\UnitDisk)$. By Ahlfors' theorem, $f:X \to \UnitDisk$ is a proper holomorphic map of degree $n$, which extends continuously to $\partial X$ by the Schwarz reflection principle. Let $\alpha_1<\beta_1<\dots<\alpha_n<\beta_n$ be the intersection points of $\partial X$ with $\Reals$. We will prove that $f(\beta_j)=1$ for each $j$.

First note that $X=R^{-1}(\UnitDisk)$ is symmetric with respect to the real axis since $\UnitDisk$ is and $\overline{R(\bar z)}= R(z)$. Moreover, since $R$ is $n$-good, each component of $\partial X$ is a Jordan curve symmetric about the real axis. It follows that $X \cap \mathbb{H}$ is a Jordan domain, where $\mathbb{H}$ is the upper half-plane. Let $\psi : \overline{X \cap \mathbb{H}} \to \overline{\mathbb{H}}$ be a homeomorphism conformal on $X \cap \mathbb{H}$ with $\psi(\infty)=\infty$. We can extend $\psi$ to all of $X$ by the Schwarz reflection principle. The resulting map $\varphi$ is a biholomorphism of $X$ onto the complement $Y$ in $\RiemannSphere$ of $n$ disjoint closed intervals in the real line. Since $\psi$ preserves the orientation on the boundary, we see that $\varphi$ is strictly increasing on $X\cap \Reals$, and in particular $\lim_{z\to \infty}\varphi(z)/z$ is positive.

If $g$ denotes the Ahlfors function on $Y$, then $f=g \circ \varphi$ by the transformation law. By Lemma \ref{lem2}, $g$ is real-valued and strictly decreasing on $Y\cap \Reals$, and thus $f$ is real-valued and strictly decreasing on $X\cap \Reals$. Since $f(\partial X \cap \Reals) \subset \partial \mathbb{D} \cap \Reals = \{ -1, 1 \}$, we must have $f(\alpha_j)=-1$ and $f(\beta_j)=1$ for each $j\in \{1, ..., n \}$.

Now, $R : X \to \UnitDisk$ is also a proper holomorphic map of degree $n$ and we have $f(\infty) = R(\infty) = 0$ and $f(\beta_j)=R(\beta_j)=1$ for each $j$. By Theorem \ref{normalizedBieberbach}, $f=R$.
\end{proof}

There is a partial converse to Theorem \ref{theo real} :

\begin{theorem}
\label{posres}
Let $R$ be an $n$-good rational map with all real poles and real residues. If $R$ is a rational Ahlfors function, then all of its residues are positive.
\end{theorem}
\begin{proof}
Let $X := R^{-1}(\UnitDisk)$ and let $p_1< \dots <p_n$ be the poles of $R$. Let us prove first that $R$ does not have any critical point on $X \cap \mathbb{R}$.

The assumptions on the poles and the residues imply that $\overline{R(\bar z)}=R(z)$, so that $X$ is symmetric about the real axis. As in the proof of Theorem \ref{theo real}, there exists a biholomorphism $\varphi$ of $X$ onto the complement $Y$ in $\RiemannSphere$ of a finite number of disjoint closed intervals in the real line, with $\varphi(\infty)=\infty$ and the derivative of $\varphi$ strictly positive on $X \cap \Reals$.

Let $g$ be the Ahlfors function on $Y$. We have $R=g \circ \varphi$ by the transformation law. By Lemma \ref{lem2}, $g$ does not have any critical point on $Y \cap \Reals$, and thus $R$ does not have any critical point on $X \cap \Reals$.

Since $R$ is the Ahlfors function on $X$, its derivative at infinity $R'(\infty)$ is positive. Observe that this derivative is equal to $\lim_{z \to \infty} z R(z) = \sum_{j=1}^n a_j,$
the sum of the residues of $R$, so at least one of them must be positive. Suppose that there is also one negative residue. Then there must be two consecutive poles $p_j$ and $p_{j+1}$ whose corresponding residues have opposite signs. Hence $R(t)$ approaches $+\infty$, or $-\infty$, in the same way as $t$ decreases to $p_j$ or increases to $p_{j+1}$. This implies that $R$ must have a critical point in the interval $(p_j, p_{j+1})$. By Lemma \ref{nconnected}, all the critical points of $R$ are in $X$, so $R$ has a critical point in $X \cap \Reals$, a contradiction. Therefore, all the residues of $R$ are positive.

\end{proof}

We do not know if Theorem \ref{theo real} holds if $R$ is allowed to have conjugate pairs of poles.

\begin{question}
\label{question symdom}
Suppose that an $n$-good rational map $R$ satisfies $R(\overline z) = \overline{R(z)}$ and has all positive residues. Then is $R$ a rational Ahlfors function?
\end{question}

\section{Rotational symmetry}
\label{sec3}

In this second family of examples, the domain $R^{-1}(\UnitDisk)$ admits a rotational symmetry which allows us to show that the Ahlfors function $f$ on $R^{-1}(\UnitDisk)$ has the same zeros as $R$, from which it follows that $f=R$.

\begin{theorem}
\label{rotsym}
Let $n\geq 2$, $0<a < n (n-1)^{(1-n)/n}$, and $R(z)=a z^{n-1} / (z^n - 1)$. Then $R$ is a rational Ahlfors function.
\end{theorem}
\begin{proof}
We have
$$
R'(z) = -a z^{n-2} \frac{z^{n}+(n-1)}{(z^n-1)^2}.
$$
Let $q= e^{i\pi / n}$ and $\omega = q^2$.
The critical points of $R$ are $$(n-1)^{1/n}q,\, \omega(n-1)^{1/n}q,\, \ldots,\, \omega^{n-1}(n-1)^{1/n}q$$ and $0$ with multiplicity $(n-2)$. We have $R(0)=0$ and $$|R(\omega^k(n-1)^{1/n}q)|=a (n-1)^{(n-1)/n} / n < 1,$$
so by Lemma \ref{nconnected}, $R^{-1}(\UnitDisk)$ is connected and bounded by $n$ disjoint analytic Jordan curves.

Let $X := R^{-1}(\UnitDisk)$ and let $f$ be the Ahlfors function on $X$. By Ahlfors' theorem, $f$ is a proper holomorphic map of degree $n$ and extends continuously to $\partial{X}$.

We want to show that $f$ has the same zeros as $R$, that is, a simple zero at infinity and a zero of multiplicity $(n-1)$ at the origin.

We have that $\omega R (\omega z) = R(z)$ and hence $\omega X = X$.  This implies that $$\omega f(\omega z) = f(z),$$ since the left-hand side maps $X$ holomorphically into $\UnitDisk$ and has the same derivative as $f$ at infinity. Suppose that $f$ vanishes at a point $z$ in $X$ different from $0$ or $\infty$. Then $f$ vanishes at $z,\omega z, \ldots, \omega^{n-1} z$ and $\infty$, a total of $n+1$ points, which is impossible since it has degree $n$. Thus $f$ vanishes only at $0$ and $\infty$. Since $f'(\infty)>0$, the zero at infinity is simple and hence the zero at the origin has multiplicity $(n-1)$.

Therefore, $f/R$ and $R/f$ are holomorphic in $X$. Since $|f/R|=1$ on all $\partial X$, the maximum principle applied to both quotients implies that $|f/R|=1$ on $X$. In particular, $f/R$ does not have open image and is thus a unimodular constant. Since $f'(\infty)$ and $R'(\infty)=a$ are both positive, $f=R$.
\end{proof}

\begin{remark}
The map $R(z)=a z^{n-1}/(z^n - 1)$ can alternatively be written as
$$
R(z)=\frac{a}{n}\sum_{j= 1}^n \frac{1}{z-\omega^j},
$$
where $\omega = e^{2\pi i / n}$, so that $R$ has the same positive residue $a/n$ at each of its poles.
\end{remark}

\begin{remark}
For $n=2$, the above shows that $R(z)= a z / (z^2-1) = \frac{a}{2}\left(\frac{1}{z+1}+\frac{1}{z-1}\right)$ is a rational Ahlfors function for all $a\in(0,2)$. Let $M(z)= (z+i)/(z-i)$. By an analogue of the transformation law, the composition
$$i R \circ M (z) =\frac{a}{4}(z+1/z)=: J(z)$$
--- the Joukowsky transform --- is the Ahlfors function on $J^{-1}(\UnitDisk)$ with respect to the point $i=M^{-1}(\infty)$. This was proved in \cite{BELL5} by much less elementary means.
\end{remark}

\section{Degree two}

We now determine all rational Ahlfors functions of degree two.

\begin{reptheorem}{deg2}
A rational map of degree two is a rational Ahlfors function if and only if it can be written in the form
$$R(z)= \frac{a_1}{z-p_1}+\frac{a_2}{z-p_2},$$
for distinct $p_1, p_2 \in \mathbb{C}$ and positive $a_1,a_2$ satisfying $a_1+a_2 < |p_1-p_2|$.
\end{reptheorem}

\begin{proof}
We will prove that a rational map of degree two with poles at $0$ and $1$ is a rational Ahlfors function if and only if it can be written in the form
$$R(z)= \frac{a_1}{z}+\frac{a_2}{z-1},$$
for positive $a_1,a_2$ satisfying $a_1+a_2 < 1$. This is sufficient by Observation \ref{affine} and the fact that any two distinct points in the plane can be moved to $0$ and $1$ by a complex affine transformation.

\paragraph{$(\Leftarrow)$} The critical points of $R$ are the solutions to the equation
$$
a_1(z-1)^2 + a_2 z^2 =0,
$$
which are given by
$$
\frac{a_1 \pm i \sqrt{a_1 a_2}}{a_1 + a_2}.
$$

One computes that the corresponding critical values are
$$
(a_1-a_2)\mp i 2\sqrt{a_1 a_2}
$$
and that their modulus is $a_1+a_2$, which is less than $1$ by hypothesis.

By Lemma \ref{nconnected}, $R^{-1}(\UnitDisk)$ is doubly connected, and by Theorem \ref{theo real}, $R$ is Ahlfors on $R^{-1}(\UnitDisk)$.

\paragraph{$(\Rightarrow)$} Suppose that $R$ is a rational Ahlfors function of degree two with poles at $0$ and $1$ and write $R(z)=a_1 / z + a_2 / (z-1)$ for some $a_1, a_2 \in \ComplexPlane \setminus\{ 0 \}$. We want to prove that the residues $a_1$ and $a_2$ are real and positive.

We start by proving that $a_1$ and $a_2$ are real, or that $R$ is equal to
$$Q(z):=\overline{R(\overline{z})}=\frac{\overline{a_1}}{z}+\frac{\overline{a_2}}{z-1}.$$

We claim that there is a biholomorphism $g : R^{-1}(\UnitDisk) \to Q^{-1}(\UnitDisk)$ satisfying $g(\infty)=\infty$ and $\lim_{z \to \infty} g(z)/z = \lambda $ for some $\lambda \in \ComplexPlane$ with $|\lambda|=1$.

Indeed, by hypothesis, $R^{-1}(\UnitDisk)$ is a doubly connected domain, and as such it is bi\-holomorphic to a circular annulus. It follows that there is an anticonformal involution $J : R^{-1}(\UnitDisk) \to R^{-1}(\UnitDisk)$ fixing $\infty$ as well as the two boundary components. As the complex conjugation maps $R^{-1}(\UnitDisk)$ onto $Q^{-1}(\UnitDisk)$, $g(z) = \overline{J(z)}$ is the biholomorphism sought-after. The limit $\lambda=\lim_{z\to \infty} \overline{J(z)} / z$ has modulus $1$ because $J\circ J = \id$.

Consider the map $\lambda Q \circ g$. This is a holomorphic map from $R^{-1}(\UnitDisk)$ to $\UnitDisk$ with derivative
$$
(\lambda Q \circ g)'(\infty)=Q'(\infty)=\overline{R'(\infty)}=R'(\infty)
$$
at infinity. The last equality holds because $R$ is the Ahlfors function on $R^{-1}(\UnitDisk)$, so that $R'(\infty)$ is positive. By uniqueness of the Ahlfors function, we have
$$
R = \lambda Q \circ g.
$$

Writing this equation as $R \circ \id = (\lambda Q) \circ g$, we see from the uniqueness part of Theorem \ref{Bellrep} that $g$ is the restriction of a M\"obius transformation for which the equality $R = \lambda Q \circ g$ holds on all of $\RiemannSphere$. By construction, $g$ maps each boundary component of $\partial R^{-1}(\UnitDisk)$ to its image by complex conjugation, and it follows that $g(0)=0$ and $g(1)=1$. Since we also have $g(\infty)=\infty$, $g$ is the identity map. Thus $\lambda=1$, and $R=Q$.

This proves that the residues $a_1$ and $a_2$ of $R$ are real. By Theorem \ref{posres}, they must be positive. The sum $a_1+a_2$ has to be less than $1$ for $R^{-1}(\UnitDisk)$ to be doubly connected, as shown by the calculation carried out in the first part of the proof.
\end{proof}

\section{Computation of analytic capacity and numerical examples}

\subsection{Analytic capacity}

Let $K$ be a compact subset of $\mathbb{C}$, and let $X$ be the component of $\RiemannSphere \setminus K$ containing $\infty$. The \textit{analytic capacity of} $K$ is defined as
$$\gamma(K):= \max \{\, \Re f'(\infty) \mid f: X \to \mathbb{D} \mbox{ is holomorphic with } f(\infty)=0 \,\},$$
the derivative at infinity of the Ahlfors function on $X$.

Historically, analytic capacity has been studied in connection with a problem of Painlev\'e, on finding a geometric characterization of the compact sets which are removable for bounded holomorphic functions. Analytic capacity has also been linked in the 1960's by Vitushkin to the theory of uniform rational approximation of holomorphic functions.

In general, it is quite difficult to identify precisely the analytic capa\-city or the Ahlfors function, even for nice compact sets such as finite unions of disjoint round disks. Moreover, despite recent advances due to Tolsa and many others, the pro\-perties of analytic capacity are far from being fully understood. For example, it is not known whether analytic capacity is subadditive.

On the other hand, the numerical computation of the analytic capacity of compact sets bounded by analytic Jordan curves is now rendered possible, thanks to a method from \cite{YOU}. In particular, this method applies to the compact set $R^{-1}(\RiemannSphere\setminus \UnitDisk)$ for any given $n$-good rational map $R$, and allows us to check whether or not $R$ is a rational Ahlfors function.

\subsection{Computation of analytic capacity}
\label{subsec31}
Suppose that $X$ is a domain containing $\infty$ bounded by a finite number of disjoint analytic Jordan curves, let $K:=\RiemannSphere \setminus X$, and let $A(X)=C(\overline{X})\cap\mathcal{O}(X)$ be the set of all functions continuous on $\overline{X}$ and holomorphic on $X$. Then we have the following :

\begin{theorem}[Garabedian \cite{GARA}]
\label{est1}
\begin{equation*}
\gamma(K) = \min \left\{\, \frac{1}{2\pi}\int_{\partial X} |g(z)|^2|dz| : g \in A(X), \, g(\infty)=1 \,\right\}.
\end{equation*}
\end{theorem}

\begin{theorem}[Younsi--Ransford \cite{YOU}]
\label{est2}
\begin{equation*}
\gamma(K) = \max \left\{\, 2\Re h'(\infty) - \frac{1}{2\pi}\int_{\partial X} |h(z)|^2 |dz| : h \in A(X), \, h(\infty)=0 \,\right\}.
\end{equation*}

\end{theorem}

We also mention that Theorem \ref{est1} was extended to arbitrary compact sets by Havinson \cite{HAV} and that Theorem \ref{est2} remains true under the weaker assumption that $K$ is bounded by piecewise-analytic curves, provided $A(X)$ is replaced by the Smirnov class $E^2(X)$. %See \cite{YOU} for more details.

These estimates yield a practical method for the computation of $\gamma(K)$, based on quadratic form minimization. We explain the method for obtaining a decreasing sequence of upper bounds for $\gamma(K)$, following \cite{YOU}. The method for obtaining lower bounds is similar.

Let $S$ be any finite set containing at least one point in each component of the interior of $K$. The set of rational functions with poles in $S$ is uniformly dense in $A(X)$, by Mergelyan's theorem (although it has been known in this special case since Runge's theorem).

For each $k \in \mathbb{N}$, we let $\mathcal{F}_k$ be the set of all functions of the form $(z-p)^{-j}$, with $1 \leq j \leq k$ and $p\in S$. In view of Theorem \ref{est1}, the quantity
$$
u_k := \min \left\{\, \frac{1}{2\pi}\int_{\partial X} |1+g(z)|^2|dz| : g \in \spann \mathcal{F}_k \,\right\}
$$
gives an upper bound for $\gamma(K)$.

In practice, we find the minimum $u_k$ as follows.

\begin{enumerate}
\item We define
$$
g(z):= \sum_{g_j \in \mathcal{F}_k} (c_j + id_j) g_j(z),
$$
where the $c_j$'s and $d_j$'s are real numbers to determine.

\item We compute the integral
$$
\frac{1}{2\pi}\int_{\partial X} |1+g(z)|^2|dz|
$$
as an expression in the $c_j$'s and $d_j$'s. This gives a quadratic form with a linear term and a constant term.

\item We find the $c_j$'s and $d_j$'s that minimize this expression, by solving the corres\-ponding linear system.
\end{enumerate}

The sequence $(u_k)_{k=1}^\infty$ is decreasing by construction, and it converges to $\gamma(K)$ by Mergelyan's theorem.

\subsection{Numerical examples}

In this subsection, we present several numerical exam\-ples where we compute the analytic capacity of compact sets of the form $$K=R^{-1}(\RiemannSphere \setminus \UnitDisk)=\{\,z \in \mathbb{C} : |R(z)| \geq 1 \,\},$$ where $R$ is an $n$-good rational map.

Observe that if $R(z) = \sum_{j=1}^n a_j/(z-p_j),$ then $R'(\infty) = \sum_{j=1}^n a_j,$ and hence $R$ is a rational Ahlfors function if and only if $\gamma(K)= \sum_{j=1}^{n}a_j$.

The method for the computation of $\gamma(K)$ described in the previous subsection involves integrals over the boundary of $K$ with respect to arclength measure. Consequently, the first step is to obtain a parametrization of $\partial K = \{\,z \in \mathbb{C} : |R(z)| = 1 \,\}$. We do this by solving for $z$ in the equation $R(z) = e^{it}$, which is equivalent to finding the roots of a polynomial of degree at most $n$. This is easily done with \textsc{maple} when $n\leq 3$.

The second step is to choose a set $S$ containing at least one point in each component of the interior of $K$. In each example, we simply take $S=\{p_1,...,p_n\}$, the set of poles of $R$.

Finally, we have to compute different integrals over $\partial K$ and solve a linear system. This numerical work is done with \textsc{matlab}, the integrals being computed with a precision of $10^{-9}$.

We first test the method on rational $n$-good maps which we know are rational Ahlfors functions by Theorem \ref{theo real} and Theorem \ref{rotsym}. Then we give an example which supports Question \ref{question symdom}. Finally, we give an example of a $3$-good rational map with positive residues which is not a rational Ahlfors function.

\newpage

\begin{example}
\label{ex1}
Let
$$R(z)=  \frac{0.3}{z+1}+ \frac{0.2}{z-1}.$$
Below is a graph representing the boundary of the compact set $$K=\{\,z \in \mathbb{C}: |R(z)|\geq 1 \,\}.$$

\begin{figure}[h!t!b]
\begin{center}
\includegraphics[width=6cm, height=6cm]{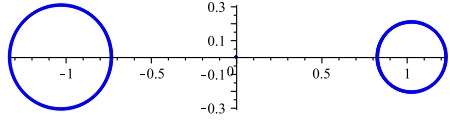}
\caption{The boundary of the compact set $K$ for Example \ref{ex1}}
\end{center}
\end{figure}

Table 1 contains the bounds for $\gamma(K)$ obtained with the method.

\begin{table}[!hbp]
\begin{center}
\caption{Lower and upper bounds for $\gamma(K)$ for Example \ref{ex1}}
\label{table1}
\begin{tabular}{|c|l|l|}
\hline
$k$ & Lower bound for $\gamma(K)$ & Upper bound for $\gamma(K)$ \\
\hline
$1$ & 0.492562045464946 &  0.500047419736669 \\
$2$ & 0.499952584760167 &  0.500003281768904 \\
$3$ &  0.499996718252636 &  0.500000110442346 \\
$4$ & 0.499999889557678 &  0.500000003956031 \\
$5$ & 0.499999996043969 &  0.500000000292436\\
\hline
\end{tabular}
\end{center}
\end{table}

By Theorem \ref{theo real}, $R$ is the Ahlfors function for the compact set $K$, and thus
$$\gamma(K) = R'(\infty) = 0.3+0.2=0.5.$$
Hence we see that the results obtained numerically are consistent with the predicted value.
\end{example}

\newpage

\begin{example}
\label{ex2}

Let
$$R(z)=  \frac{0.95}{z+1} + \frac{0.98}{z-1}.$$
Below is a graph representing the boundary of the compact set $$K=\{\,z \in \mathbb{C}: |R(z)|\geq 1 \,\},$$ consisting of two disjoint non-convex analytic Jordan curves.

\begin{figure}[h!t!b]
\begin{center}
\includegraphics[width=6cm, height=6cm]{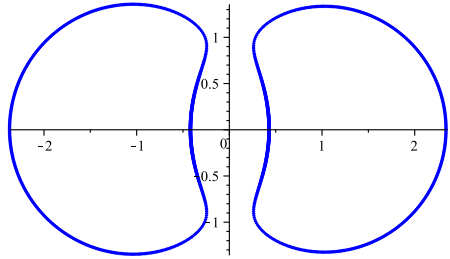}
\caption{The boundary of the compact set $K$ for Example \ref{ex2}}
\end{center}
\end{figure}

Table 2 contains the bounds for $\gamma(K)$ obtained with the method.

\begin{table}[!hbp]
\begin{center}
\caption{Lower and upper bounds for $\gamma(K)$ for Example \ref{ex2}}
\label{table2}
\begin{tabular}{|c|l|l|}
\hline
$k$ & Lower bound for $\gamma(K)$ & Upper bound for $\gamma(K)$ \\
\hline
$1$ &  1.469145654305464 &  1.998883274734441 \\
$2$ & 1.863490503463674 &  1.997657625980182 \\
$3$ &  1.864633834925701 &  1.957570768708159 \\
$4$ & 1.902817542815138 &  1.956984859867938 \\
$5$ &  1.903387234304595 &  1.944734961210238\\
$10$ & 1.924138647216576 & 1.935693736889831\\
$20$ & 1.928820666790728 & 1.931123140362772\\
$30$ & 1.929615838914482 &  1.930334911010434\\
$35$ & 1.929706466138935 & 1.930230869959049\\
$40$ & 1.929751020215389 & 1.930091261090859\\

\hline
\end{tabular}
\end{center}
\end{table}

By Theorem \ref{theo real}, $R$ is the Ahlfors function for the compact set $K$, and thus
$$\gamma(K) = R'(\infty) = 0.95+0.98=1.93.$$
Once again we see that the results obtained numerically are consistent with the predicted value, even though the convergence is slower in this case.
\end{example}

\newpage

\begin{example}
\label{ex3}

$$R(z) = \frac{0.2}{z+2} + \frac{0.1}{z} + \frac{0.4}{z-5}.$$
Below is a graph representing the compact set $$K=\{\,z \in \mathbb{C}: |R(z)|\geq 1 \,\}.$$

\begin{figure}[h!t!b]
\begin{center}
\includegraphics[width=6cm]{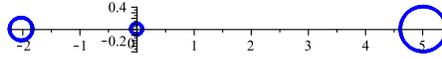}
\caption{The compact set $K$ for Example \ref{ex3}}
\end{center}
\end{figure}

Table 3 contains the bounds for $\gamma(K)$ obtained with the method.

\begin{table}[!hbp]
\begin{center}
\caption{Lower and upper bounds for $\gamma(K)$ for Example \ref{ex3}}
\label{table3}
\begin{tabular}{|c|l|l|}
\hline
$k$ & Lower bound for $\gamma(K)$ & Upper bound for $\gamma(K)$ \\
\hline
$1$ &   0.696735209508754 &     0.700011861859377\\
$2$ &    0.699988138057939 &       0.700000163885012 \\
$3$ &     0.699999835775098 &     0.700000002518033 \\
\hline
\end{tabular}
\end{center}
\end{table}

By Theorem \ref{theo real}, $R$ is the Ahlfors function for the compact set $K$, and again we have
$$\gamma(K) = R'(\infty) = 0.2+0.1+0.4=0.7.$$
\end{example}

\newpage

\begin{example}
\label{ex4}
Let
$$R(z) = \frac{z^2}{z^3-1}.$$

Below is a graph representing the boundary of the compact set $$K=\{\,z \in \mathbb{C}: |R(z)|\geq 1 \,\}.$$

\begin{figure}[h!t!b]
\begin{center}
\includegraphics[width=6cm, height=6cm]{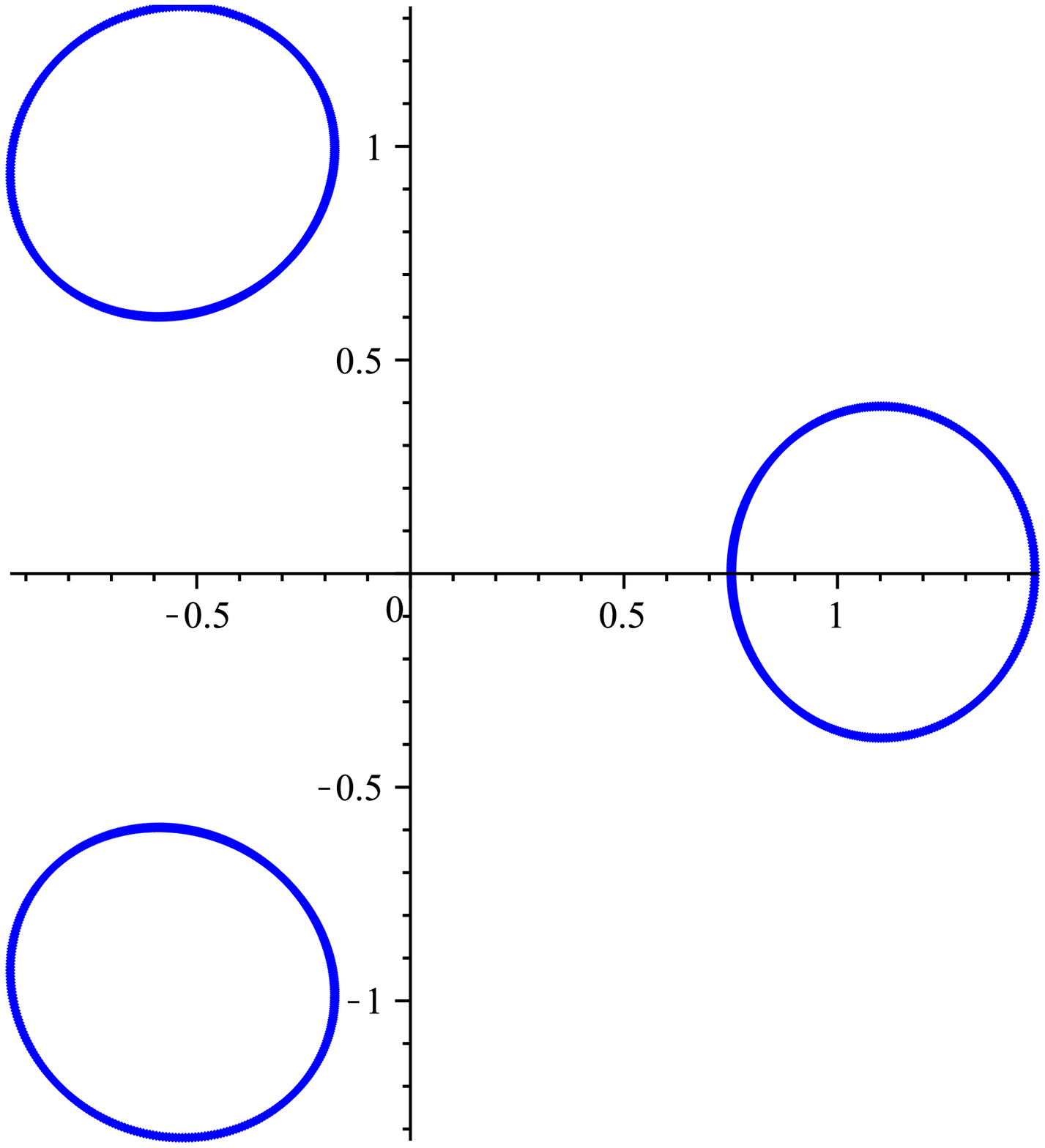}
\caption{The boundary of the compact set $K$ for Example \ref{ex4}}
\end{center}
\end{figure}

Table 4 contains the bounds for $\gamma(K)$ obtained with the method.

\begin{table}[!hbp]
\begin{center}
\caption{Lower and upper bounds for $\gamma(K)$ for Example \ref{ex4}}
\label{table4}
\begin{tabular}{|c|l|l|}
\hline
$k$ & Lower bound for $\gamma(K)$ & Upper bound for $\gamma(K)$ \\
\hline
$1$ &   0.897012961211562 &      1.003766600572323\\
$2$ &    0.996247533470256&       1.000449247199905 \\
$3$ &     0.999550954532515 &     1.000227970885994 \\
$4$ &     0.999772081072887 &      1.000015305500631 \\
$5$ &     0.999984694733624 &     1.000004234543914 \\
$6$ &     0.999995765474017 &      1.000002049275081 \\
\hline
\end{tabular}
\end{center}
\end{table}

By Theorem \ref{rotsym}, $R$ is the Ahlfors function for the compact set $K$, and thus we have
$$\gamma(K) = R'(\infty) = 1.$$
\end{example}

\newpage
\begin{example}
\label{ex5}
The following example is in connection with Question \ref{question symdom}. Let
$$R(z) = \frac{0.5}{z} + \frac{0.4}{z-(2+i)} + \frac{0.4}{z-(2-i)},$$
so that $R$ has positive residues and satisfies $R(\overline z) = \overline{R(z)}$ for all $z$.

Below is a graph representing the boundary of the compact set $$K=\{\,z \in \mathbb{C}: |R(z)|\geq 1 \,\}.$$

\begin{figure}[h!t!b]
\begin{center}
\includegraphics[width=6cm, height=6cm]{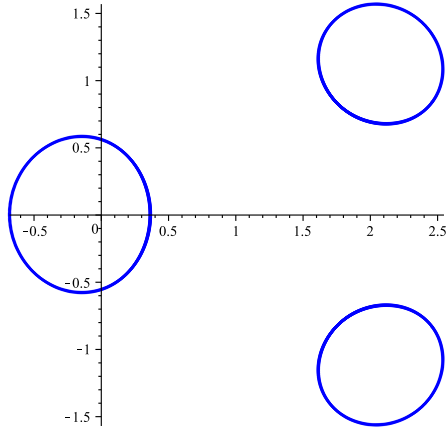}
\caption{The boundary of the compact set $K$ for Example \ref{ex5}}
\end{center}
\end{figure}

Table 5 contains the bounds for $\gamma(K)$ obtained with the method.

\begin{table}[!hbp]
\begin{center}
\caption{Lower and upper bounds for $\gamma(K)$ for Example \ref{ex5}}
\label{table5}
\begin{tabular}{|c|l|l|}
\hline
$k$ & Lower bound for $\gamma(K)$ & Upper bound for $\gamma(K)$ \\
\hline
$1$ &   1.156483451112665 &    1.306262607579208 \\
$2$ &    1.293906716808142 &      1.300866124135705 \\
$3$ &     1.299286594644695 &     1.300451765037035 \\
$4$ &   1.299697642245979 &    1.300120036845019 \\

\hline
\end{tabular}
\end{center}
\end{table}

The results obtained seem to indicate that $\gamma(K)=1.3=R'(\infty)$ or equivalently, that $R$ is a rational Ahlors function. This was also the case for every example of symmetric rational function that we tested numerically. The answer to Question \ref{question symdom} thus appears to be positive, at least in degree $3$.
\end{example}

\newpage
\begin{example}
\label{ex6}
Our last example shows that the positivity of the residues of a $3$-good rational map is not sufficient for it to be a rational Ahlfors function. Let

$$R(z) = \frac{0.4}{z} + \frac{0.4}{z-6} + \frac{0.4}{z-(1+i)}.$$

Below is a graph representing the boundary of the compact set $$K=\{\,z \in \mathbb{C}: |R(z)|\geq 1 \,\}.$$

\begin{figure}[h!t!b]
\begin{center}
\includegraphics[width=6cm, height=6cm]{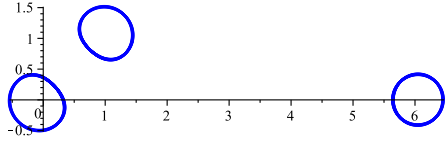}
\caption{The boundary of the compact set $K$ for Example \ref{ex6}}
\end{center}
\end{figure}

Table 6 contains the bounds for $\gamma(K)$ obtained with the method.

\begin{table}[!hbp]
\begin{center}
\caption{Lower and upper bounds for $\gamma(K)$ for Example \ref{ex6}}
\label{table6}
\begin{tabular}{|c|l|l|}
\hline
$k$ & Lower bound for $\gamma(K)$ & Upper bound for $\gamma(K)$ \\
\hline
$1$ &   1.125853723035751 &    1.203267502101022  \\
$2$ &    1.197416632904951 &       1.201353200697178 \\
$3$ &     1.199380567900335 &     1.200524665448821 \\
$4$ &   1.200219059439418 &   1.200426277666660 \\
$5$ &   1.200321460719667 &   1.200387399481300  \\
$6$ &    1.200361472698255 &    1.200378783416171 \\
$7$ &    1.200370456320151 &     1.200375934512287 \\

\hline
\end{tabular}
\end{center}
\end{table}

This table shows that $R$ is not a rational Ahlfors function, since
$$R'(\infty) = 1.2 < 1.200370456320151 \leq \gamma(K).$$
This phenomenon seems to be generic. We tested several $3$-good rational maps with positive residues, and whenever the domain $R^{-1}(\UnitDisk)$ admitted no symmetry, we had $R'(\infty) < \gamma(K)$. The example above was the one we found with the largest ratio $\gamma(K) / R'(\infty)$.

\end{example}

\newpage

\subsection{Generalizing to higher degrees}

In this subsection, we show how to gene\-ralize Example \ref{ex6} to obtain examples with higher degree.

\begin{theorem} \label{generalizedexample}
For every $n\geq 3$, there exists an $n$-good rational map which has only positive residues and is not Ahlfors.
\end{theorem}

\begin{proof}
For $n=3$, such a map is given by
$$
R(z) = \frac{0.4}{z} + \frac{0.4}{z-6} + \frac{0.4}{z-(1+i)}.
$$

For $n>3$, choose any distinct points $b_4,...,b_n$ in $R^{-1}(\UnitDisk)$. For $\varepsilon>0$, define
$$
R_\varepsilon(z) := R(z) + \sum_{j=4}^n \frac{\varepsilon}{z-b_j}.
$$
Our claim is that when $\varepsilon$ is small enough, $R_\varepsilon$ is $n$-good and is not Ahlfors.

We have that $R_\varepsilon \to R$ locally uniformly on $\RiemannSphere \setminus \{ b_4 ,..., b_n \}$ as $\varepsilon \to 0$. This implies that $R_\varepsilon^{-1}(\RiemannSphere \setminus \UnitDisk)$ converges in the sense of Hausdorff to $R^{-1}(\RiemannSphere \setminus \UnitDisk) \cup \{ b_4 ,..., b_n\}$ as $\varepsilon \to 0$. In particular, when $\varepsilon$ is small enough, $R_\varepsilon^{-1}(\RiemannSphere \setminus \UnitDisk)$ has at least $n$ connected components, and thus $R_\varepsilon$ is $n$-good.

Let $f$ be the Ahlfors function on $R^{-1}(\UnitDisk)$ and let $f_\eps$ be the Ahlfors function on $R_\eps^{-1}(\UnitDisk)$. By Ahlfors' theorem, $f$ extends to a bounded holomorphic function on some neighborhood $U$ of $\overline{R^{-1}(\UnitDisk)}$. Let $r>1$, and let $U_r:=f^{-1}(r \UnitDisk)$. If $\eps$ is small enough, then $U_r$ contains $R_\varepsilon^{-1}(\UnitDisk)$, so that we have
$$
f'(\infty) \leq r \, f_\eps'(\infty).
$$
This inequality holds because $(1/r) f$ is a holomorphic function on $R_\varepsilon^{-1}(\UnitDisk)$ with image in $\UnitDisk$.

We thus have
$$
f'(\infty) \leq r \liminf_{\eps \to 0} f_\eps'(\infty).
$$
Since $r$ can be taken arbitrarily close to $1$, we obtain
$$
f'(\infty) \leq \liminf_{\eps \to 0} f_\eps'(\infty).
$$

Now suppose that there is a sequence $\eps_k \to 0$ such that $R_{\eps_k}=f_{\eps_k}$ for every $k$. The above inequality yields
$$
f'(\infty) \leq \liminf_{k \to \infty} f_{\eps_k}'(\infty) = \liminf_{k \to \infty} R_{\eps_k}'(\infty) = R'(\infty).
$$
This contradicts the lower bound $f'(\infty) \geq 1.200370456320151 > 1.2 = R'(\infty)$ obtained in Table \ref{table6}.

Therefore, when $\eps$ is small enough, $R_\eps$ is an $n$-good map with all positive residues which is not Ahlfors.
\end{proof}

\section{Residues need not be positive}

The goal of this last section is to prove that the residues of a rational Ahlfors function are not necessarily positive.

\begin{theorem}\label{nonnecessity}
For every $n\geq 3$, there exists a rational Ahlfors function of degree $n$ whose residues are not all positive.
\end{theorem}

We were unable to find any explicit examples of rational Ahlfors functions with non-positive residues, and for this reason our proof proceeds by contradiction.

Fix some $n \geq 3$. By Ahlfors' theorem and Corollary \ref{normalizedBellrep}, there is a map $A$ from non-degenerate $n$-connected domains containing $\infty$ to $n$-good maps, whose image is precisely the set of rational Ahlfors functions of degree $n$. The statement of Theorem \ref{nonnecessity} is equivalent to saying that the image of $A$ is not contained in the set of $n$-good maps with positive residues.

To show this, we need to think of the domain and range of the map $A$ as mani\-folds. In order to parametrize the set of isomorphism classes of non-degenerate $n$-connected domains containing infinity, we appeal to the following theorem of Koebe.

\begin{theorem}[Koebe]
Let $X$ be a non-degenerate $n$-connected domain containing $\infty$. Then there exists a unique injective meromorphic map $h : X \to \RiemannSphere$ normalized at infinity such that the complement $\RiemannSphere \setminus h(X)$ is a union of $n$ disjoint round disks.
\end{theorem}

Accordingly, we let $\mathcal M$ denote the set of $(c,r) \in \ComplexPlane^n \times\Reals_+^n$ such that the closed disks $\overline{\UnitDisk}(c_j,r_j)$ are pairwise disjoint. It is easy to see that $\mathcal{M}$ is open and hence a manifold of dimension $3n$.

Next, let $\mathcal R$ denote the set of $(a,p) \in (\ComplexPlane\setminus \{ 0 \})^n \times \ComplexPlane^n $ such that the rational map
$$
R(z)=\sum_{j=1}^n \frac{a_j}{z-p_j}
$$
is $n$-good, meaning that $R^{-1}(\UnitDisk)$ is $n$-connected. It follows from Lemma \ref{nconnected} that $\mathcal{R}$ is open and hence a manifold of dimension $4n$ (see \cite[Lemma 2.7]{JEONG2}).

Finally, let $\mathcal R^+$ denote the subset of $\mathcal R$ where all the residues $a_j$ are real and positive. Clearly, $\mathcal R^+$ is a closed submanifold of $\mathcal R$ of dimension $3n$.

We can now define a map $ A : \mathcal M \to \mathcal R$, which we call \textit{the Ahlfors section}, as follows. Given $(c,r) \in \mathcal M$, define $$X:= \RiemannSphere \setminus \bigcup_{j=1}^n \overline{\UnitDisk}(c_j,r_j),$$ which is a non-degenerate $n$-connected domain containing $\infty$. Then let $f$ be the Ahlfors function on $X$, and let $f = R \circ g$ be the normalized factorization given by Corollary \ref{normalizedBellrep}. Since $R$ is an $n$-good rational map, it can be written as
$$
R(z)=\sum_{j=1}^n \frac{a_j}{z-p_j}
$$
for some $p_1,...,p_n \in \ComplexPlane$ and some $a_1,...,a_n \in \ComplexPlane \setminus \{ 0 \}$, unique up to permutation. An ordering of the terms in the above sum can be specified by requiring that for each $j$, the continuous extension of the biholomorphism $g : X \to R^{-1}(\UnitDisk)$ to $\partial X$ maps the circle of center $c_j$ and radius $r_j$ onto the boundary component of $\partial R^{-1}(\UnitDisk)$ surrounding the pole $p_j$. We then set $A(c,r):=(a,p)$.

Let us assume the following results and proceed with the proof of Theorem \ref{nonnecessity}.

\begin{lemma} \label{connectedness}
The manifold $\mathcal{R}^+$ is connected.
\end{lemma}

\begin{theorem} \label{continuity}
The map $A : \mathcal M \to \mathcal R$ is continuous, injective, and has closed image.
\end{theorem}

\begin{proof}[Proof of Theorem \ref{nonnecessity}]
Suppose for a contradiction that the image of $A$ is contained in $\mathcal{R}^+$. By Theorem \ref{continuity}, the map $A : \mathcal{M} \to \mathcal{R}^+$ is a continuous injective map with closed image. Since $\mathcal{M}$ and $\mathcal{R}^+$ are both manifolds of dimension $3n$, the map $A$ is also open by Brouwer's invariance of domain. The image of $A$ is thus open and closed in $\mathcal{R}^+$. Since the latter is connected by Lemma \ref{connectedness}, $A$ surjects onto $\mathcal{R}^+$. This contradicts Theorem \ref{generalizedexample}, which implies that $\mathcal{R}^+ \setminus A(\mathcal M)$ is nonempty.
\end{proof}

We explain Lemma \ref{connectedness} and Theorem \ref{continuity} in the next two subsections.

\subsection{Connectedness}

Let
$$\mathcal{F}_n\ComplexPlane := \{\, (p_1,...,p_n) \in \ComplexPlane^n \mid p_i \neq p_j\text{ for }i\neq j \,\}$$
denote the configuration space of ordered $n$-tuples of distinct points in the plane.

For $a=(a_1,...,a_n) \in (\ComplexPlane\setminus \{ 0 \})^n$ and $p = (p_1,...,p_n) \in \mathcal{F}_n\ComplexPlane$, we introduce the notation
$$
[a/p](z) := \sum_{j=1}^n \frac{a_j}{z-p_j}.
$$

\begin{lemma} \label{epsilon}
Let $K \subset \mathcal F_n \ComplexPlane$ be a compact set. Then there exists an $\varepsilon>0$ such that if $p \in K$ and $a=(a_1,...,a_n)$ satisfies $0<|a_j|\leq\varepsilon$ for each $j\in\{1,...,n\}$, then $[a/p]$ is $n$-good.
\end{lemma}
\begin{proof}
It suffices to prove that whenever we have a sequence $(p^k)_{k=1}^\infty \subset F_n \ComplexPlane $ converging to some $p \in F_n \ComplexPlane$ and a sequence $(a^k)_{k=1}^\infty \subset (\ComplexPlane\setminus \{ 0 \})^n$ converging to $0$, then $[a^k / p^k]$ is $n$-good for all sufficiently large $k$. As $a^k \to 0$ and $p^k \to p$, the rational maps $[a^k / p^k]$ converge locally uniformly to $0$ on $\RiemannSphere \setminus \{ p_1,...,p_n \}$. This implies that the compact sets $[a^k / p^k]^{-1}(\RiemannSphere \setminus \UnitDisk)$ converge to $\{ p_1,...,p_n \}$ in the sense of Hausdorff, and it follows that $[a^k / p^k]$ is $n$-good when $k$ is large enough.
\end{proof}

We can now prove that $\mathcal R^+$ is path-connected.

\begin{proof}[Proof of Lemma \ref{connectedness}]
Let $(a^0,p^0),(a^1,p^1) \in \mathcal{R}^+$. To construct a path from $(a^0,p^0)$ to $(a^1,p^1)$, we first shrink the vector of residues $a^0$ while keeping the poles fixed. Once the residues are all small enough, we move the poles from $p^0$ to $p^1$ while adjusting the residues individually in such a way that when the poles have arrived at destination, we can just rescale the residues by a common factor to end up with $a^1$. The precise details are as follows.

It is well-known and easy to prove by induction that $\mathcal{F}_n\ComplexPlane$ is path-connected. Accordingly, let $\alpha : [0,1] \to \mathcal{F}_n\ComplexPlane$ be a path from $p^0$ to $p^1$. For each $j\in \{1,...,n\}$, we define
$$
p_j^t := \begin{cases}
p_j^0  & t \in [0,1/3] \\
\alpha(3t-1)_j & t \in (1/3,2/3) \\
p_j^1  & t\in [2/3,1].
\end{cases}
$$

Let $K:=\alpha([0,1])$ denote the trace of the path $\alpha$, and let $\eps$ be as in the previous lemma. We may assume that $\eps \leq a^i_j$ for each $i \in \{ 0, 1\}$ and each $j \in \{1,...,n\}$.

For $t \in [0,1/3]$, let $\mu_t$ be the constant which decreases linearly from $1$ to $\varepsilon / \|a^0\|_\infty$, and for $t \in [2/3,1]$ let $\nu_t$ be the constant which increases linearly from $\varepsilon / \|a^1\|_\infty$ to $1$. For each $j\in \{1,...,n\}$, we define
$$
a_j^t :=
\begin{cases}
\mu_t a_j^0  & t \in [0,1/3] \\
(2-3t)\mu_{1/3} a_j^0 + (3t-1) \nu_{2/3} a_j^1 & t \in (1/3,2/3) \\
\nu_t a_j^1 & t\in [2/3,1].
\end{cases}
$$

The path $t \mapsto (a^t,p^t)$ is clearly continuous and is such that for all $t \in [0,1]$, the poles $p_1^t,...,p_n^t$ are distinct, and the residues $a_1^t,...,a_n^t$ are positive. It remains to check that $[a^t / p^t]$ is $n$-good for each $t \in [0,1]$.

For $t\in [0,1/3]$, we have
$$
[a^t/p^t] = [\mu_t a^0/p^0] = \mu_t [a^0/p^0],
$$
with $\mu_t\in (0,1]$. Since $[a^0/p^0]$ is $n$-good by hypothesis, all of its critical values have modulus strictly less than $1$. The critical values of $\mu_t [a^0/p^0]$ have modulus strictly less than $1$ as well, since they are equal to the critical values of $[a^0/p^0]$ rescaled by the factor $\mu_t$. Therefore, $[a^t/p^t]$ is $n$-good. The same argument shows that $[a^t/p^t]=\nu_t [a^1/p^1]$ is $n$-good for every $t \in [2/3,1]$.

Finally, $[a^t/p^t]$ is $n$-good for every $t \in (1/3,2/3)$ in virtue of Lemma \ref{epsilon}, as we have $0< a_j^t \leq \eps$ for each $j$.
\end{proof}

\begin{remark}
Similar arguments show that $\mathcal M$ and $\mathcal R$ are path-connected as well.
\end{remark}

\subsection{The Ahlfors section}

It remains to explain why the Ahlfors section $A$ is continuous, injective, and has closed image.

\begin{proof}[Proof that $A$ is injective]
Let $(c^0,r^0)$ and $(c^1,r^1)$ in $\mathcal M$ be such that $A(c^0,r^0)=A(c^1,r^1)$. Let $X_0$ and $X_1$ be the circular domains corresponding to $(c^0,r^0)$ and $(c^1,r^1)$ respectively. For $k=0,1$, let $f_k$ be the Ahlfors function on $X_k$, and let $f_k = R_k \circ g_k$ be the normalized factorization of $f_k$. Since $A(c^0,r^0)=A(c^1,r^1)$, the $n$-good rational maps $R_0$ and $R_1$ have the same ordered sets of residues and poles, and are thus equal. It follows that $g_1^{-1} \circ g_0$ is a biholomorphism from $X_0$ to $X_1$. Moreover, this biholomorphism is normalized at infinity, and sends the $j$-th circle bounding $X_0$ to the $j$-th circle bounding $X_1$. By the uniqueness part of Koebe's theorem, $g_1^{-1} \circ g_0$ is the identify map. Therefore, $X_0=X_1$ and $(c^0,r^0)=(c^1,r^1)$.
\end{proof}

\begin{proof}[Proof that the image of $A$ is closed]
We first prove that the set of rational Ahlfors functions is closed in the set of $n$-good rational maps. Suppose that $R_k \to R$, where $R_k$ is a rational Ahlfors function and $R$ is an $n$-good rational map. We have to prove that $R$ is the Ahlfors function on $R^{-1}(\UnitDisk)$. Let $f$ be the Ahlfors function on $R^{-1}(\UnitDisk)$. By Ahlfors' theorem and the Schwarz reflection principle, $f$ extends to a bounded holomorphic function on some neighborhood $U$ of $\overline{R^{-1}(\UnitDisk)}$. Let $r>1$, and let $U_r = f^{-1}(r\UnitDisk)$. This is a neighborhood of $\overline{R^{-1}(\UnitDisk)}$ on which $|f| < r$. If $k$ is large enough, then $R_k^{-1}(\UnitDisk)$ is contained in $U_r$, which implies that
$$
f'(\infty) \leq r R_k'(\infty).
$$
Letting $k \to \infty$, and then $r \to 1$, we obtain $f'(\infty) \leq R'(\infty)$. By uniqueness of the Ahlfors function, $R=f$. We can now deduce the result. Suppose that $(a^k,p^k) \to (a,p)$, where $(a^k,p^k) \in A(\mathcal M)$ and $(a,p) \in \mathcal R$. Then the corresponding rational maps $[a^k/p^k]$ converge locally uniformly to $[a/p]$. By the above proof, $[a/p]$ is a rational Ahlfors function. By Koebe's theorem, there exists a biholomorphism $g$ normalized at infinity from $[a/p]^{-1}(\UnitDisk)$ to the complement $X$ of $n$ round disks in the plane. The biholomorphism $g$ extends continuously to the boundary, and maps the boundary component of $[a/p]^{-1}(\UnitDisk)$ enclosing the pole $p_j$ to some circle with center $c_j$ and radius $r_j$. By the transformation law, the Ahlfors function on $X$ is equal to $[a/p] \circ g^{-1}$. It follows that $A(c,r)=(a,p)$, and thus the image of $A$ is closed.
\end{proof}

The proof of continuity requires the following generalization of Carath\'eodory's kernel convergence theorem.

\begin{theorem}
\label{compacite}
Let $(X_k)_{k=1}^\infty$ be a sequence of domains containing $\infty$ and for each $k$, let $g_k : X_k \to  \RiemannSphere$ be injective meromorphic functions normalized at infinity. Suppose that $X_k$ converges to some domain $X$ containing $\infty$ in the sense of Carath\'eodory. Then there exists a subsequence $(g_{j_k})_{k=1}^{\infty}$ which converges locally uniformly on $X$ to an injective meromorphic function $g:X \to \RiemannSphere$ normalized at infinity. Moreover, the corresponding subsequence of images $(g_{j_k}(X_{j_k}))_{k=1}^{\infty}$ converges to $g(X)$ in the sense of Carath\'eodory and $g_{j_k}^{-1} \to g^{-1}$ locally uniformly on $g(X)$.
\end{theorem}

A proof of this theorem can be found in \cite[Theorem 1, p.228]{GOL}. It relies on the compactness of the set of injective meromorphic functions normalized at infinity on any domain containing $\infty$. The latter follows from Gronwall's area theorem, which says that if
$$
h(z) = z  + \sum_{j=1}^\infty b_j z^{-j}
$$
is injective on $\{ z \in \RiemannSphere : |z|>1 \}$, then $\sum_{j=1}^\infty j|b_j|^2 \leq 1$.

\begin{proof}[Proof that $A$ is continuous]
Suppose that $(c^k,r^k) \to (c,r)$. Let $X_k$ and $X$ be the corresponding circle domains, and let $f_k= R_k \circ g_k$ and $f=R \circ g$ be the Ahlfors functions with their normalized factorizations on $X_k$ and $X$ respectively. We first prove that $f_k \to f$ locally uniformly on $X$. Observe that
$$
f'(\infty) \leq \liminf_{k \to \infty} f_k'(\infty)
$$
since $f$ extends to a neighborhood of $\overline X$ and $\RiemannSphere \setminus X_k \to \RiemannSphere \setminus X$ in the sense of Hausdorff. By Montel's theorem, the sequence $(f_k)_{k=1}^\infty$ forms a normal family. If $\varphi$ is any limit of any subsequence of $(f_k)_{k=1}^\infty$, then $\varphi$ is a holomorphic map from $X$ to $\UnitDisk$, satisfying $\varphi(\infty)=0$ as well as $\varphi'(\infty)\geq f'(\infty)$. By uniqueness of the Ahlfors function, $\varphi = f$, and hence $f_k \to f$.

Let $S \subset \mathbb N$ be any subsequence. Then by Theorem \ref{compacite}, there exists a subsequence $S' \subset S$ such that $g_k$ converges locally uniformly on $X$ to an injective meromorphic function $h$ normalized at infinity, as $k \to \infty$ in $S'$.

Then $h(X)$ is a non-degenerate $n$-connected domain containing $\infty$. Moreover, by Theorem \ref{compacite}, $g_k(X_k) \to h(X)$ in the sense of Caratheodory and $g_k^{-1} \to h^{-1}$ locally uniformly on $h(X)$. The poles and residues of $R_k$ have to stay bounded since $R_k^{-1}(\UnitDisk)=g_k(X_k) \to h(X)$, and it follows that we may extract a further subsequence so that $R_k$ converges, outside of a finite set of points, to some rational function $Q$ of degree at most $n$. On the other hand, we have $R_k = f_k \circ g_k^{-1} \to f \circ h^{-1}$, so that $f \circ h^{-1}=Q$. Since $f : X \to \UnitDisk$ has degree $n$, the restriction $Q : h(X) \to \UnitDisk$ has degree $n$ as well. Therefore, the global degree of $Q$ is $n$ and $Q^{-1}(\UnitDisk)=h(X)$ is a non-degenerate $n$-connected domain. We thus have a factorization $f = Q \circ h$, where $Q$ is an $n$-good rational map and $h : X \to Q^{-1}(\UnitDisk)$ is a biholomorphism normalized at infinity. By Corollary \ref{normalizedBellrep}, we have $Q=R$ and $h=g$.

Since the subsequence $S$ was arbitrary, we have $R_k \to R$ and $g_k \to g$. Write $R_k=[a^k/p^k]$, where the poles are ordered in such a way that $g_k$ maps the circle with center $c_j^k$ and radius $r_j^k$ to the boundary component of $R_k^{-1}(\UnitDisk)$ surrounding $p_j^k$, and write $R=[a/p]$ similarly. As indicated earlier, the poles and residues of $R_k$ lie in a bounded set. If $(a^k,p^k) \to (a',p')$ along some subsequence, then we get that $[a^k/p^k] \to [a'/p']$ along this subsequence, and thus $[a'/p']=[a/p]$. This implies that $(a'_j,p'_j)=(a_{\sigma(j)},p_{\sigma(j)})$ for some permutation $\sigma$. However, for any simple closed curve $\alpha$ in $X$ surrounding $c_j$ and no other center, and $k$ large enough, the image $g_k(\alpha)$ is a simple closed curve surrounding $p_j^k$ and no other pole of $R_k$. The limit $p_{\sigma(j)}$ of $p_j^k$ is thus surrounded by the curve $g(\alpha)$. This implies that $\sigma$ is the identity permutation. Therefore, $(a^k,p^k) \to (a,p)$ and $A$ is continuous.
\end{proof}

\acknowledgments{The authors thank Joe Adams, Dmitry Khavinson, Thomas Ransford and the anonymous referee for helpful suggestions in order to improve the paper.}

\bibliographystyle{amsplain}

\end{document}